\documentclass[twoside,a4paper]{amsart}

\DeclareMathOperator{\arccosh}{arccosh}
\PassOptionsToPackage{pdfauthor={Vladimir V. Kisil and James Reid},%
    pdftitle={Conformal Parametrisation of Loxodromes by Triples of Circles},%
    pdfsubject={mathematics},%
    pdfkeywords={symmetries}
}{hyperref}
\usepackage[breaklinks=true,%
  colorlinks=true,%
    backref=page,%
  bookmarks=true,%
  pagebackref=true]{hyperref}
\usepackage{amssymb,graphicx,animate,amsrefs}

\IfFileExists{eulervm.sty}{%
\usepackage{eulervm}}{}

\DeclareFontFamily{OT1}{cyr}{}
\DeclareFontShape{OT1}{cyr}{m}{n}
   {  <5> <6> <7> <8> <9> gen * wncyr
      <10> <10.95> <12> <14.4> <17.28> <20.74> <24.88> wncyr10}{}
\DeclareFontShape{OT1}{cyr}{m}{it}
    {
       <5> <6> <7> <8> <9> gen * wncyi
      <10> <10.95> <12> <14.4> <17.28> <20.74> <24.88>wncyi10
      }{}
\DeclareFontShape{OT1}{cyr}{m}{ss}
    {
       <5> <6> <7> <8> wncyss8
       <9> wncy9
      <10> <10.95> <12> <14.4> <17.28> <20.74> <24.88>wncyss10
      }{}
\DeclareFontShape{OT1}{cyr}{m}{sc}
    {
       <5> <6> <7> <8> <9> <10> <10.95> <12> <14.4> <17.28> <20.74> <24.88>wncysc10
      }{}
\DeclareFontShape{OT1}{cyr}{bx}{n}
   {
       <5> <6> <7> <8> <9> gen * wncyb
      <10> <10.95> <12> <14.4> <17.28> <20.74> <24.88>wncyb10
      }{}
\DeclareTextFontCommand{\textcyr}{\fontfamily{cyr}\selectfont}
\providecommand{\cyr}{\fontfamily{cyr}\selectfont\def\cprime{\~}}
\providecommand{\cprime}{'}

\BibSpecAlias{incollection}{inproceedings}
\BibSpec{article}{%
    +{}  {\PrintAuthors}                {author}
    +{.} { }                            {title}
    +{.} { }                            {part}
    +{:} { }                            {subtitle}
    +{.} { \PrintContributions}         {contribution}
    +{.} { \PrintPartials}              {partial}
    +{.} { \emph}                       {journal}
    +{,} { \textbf}                            {volume}
    +{}  { \parenthesize}               {number}
    +{:} {}                             {pages}
    +{,} { \PrintDateB}                 {date}
    +{,} { }                            {status}
    +{.} { \PrintTranslation}           {translation}
    +{.} { Reprinted in \PrintReprint}  {reprint}
    +{.} { }                            {note}
    +{.} {}                             {transition}
}

\BibSpec{partial}{%
    +{}  {}                             {part}
    +{:} { }                            {subtitle}
    +{.} { \PrintContributions}         {contribution}
    +{.} { \emph}                       {journal}
    +{,} { \textbf}                            {volume}
    +{}  { \parenthesize}               {number}
    +{:} {}                             {pages}
    +{,} { \PrintDateB}                 {date}
}

\BibSpec{book}{%
    +{}  {\PrintPrimary}                {transition}
    +{.} { \emph}                       {title}
    +{.} { }                            {part}
    +{:} { \emph}                       {subtitle}
    +{.} { }                            {series}
    +{,} { \voltext}                    {volume}
    +{.} { Edited by \PrintNameList}    {editor}
    +{.} { Translated by \PrintNameList}{translator}
    +{.} { \PrintContributions}         {contribution}
    +{.} { }                            {publisher}
    +{.} { }                            {organization}
    +{,} { }                            {address}
    +{,} { \PrintEdition}               {edition}
    +{,} { \PrintDateB}                 {date}
    +{.} { }                            {note}
    +{.} {}                             {transition}
    +{.} { \PrintTranslation}           {translation}
    +{.} { Reprinted in \PrintReprint}  {reprint}
    +{.} {}                             {transition}
}

\BibSpec{collection.article}{%
    +{}  {\PrintAuthors}                {author}
    +{.} { }                            {title}
    +{.} { }                            {part}
    +{:} { }                            {subtitle}
    +{.} { \PrintContributions}         {contribution}
    +{.} { \PrintConference}            {conference}
    +{.} { \PrintBook}                  {book}
    +{.} { In \PrintEditorsA}              {editor}
    +{.} { \emph}                         {booktitle}
    +{,} { pages~}                      {pages}
    +{,} { }                            {publisher}
    +{,} { }                            {organization}
    +{,} { }                            {address}
    +{,} { \PrintDateB}                 {date}
    +{.} { \PrintTranslation}           {translation}
    +{.} { Reprinted in \PrintReprint}  {reprint}
    +{.} { }                            {note}
    +{.} {}                             {transition}
}

 \BibSpec{inproceedings}{%
     +{}  {\PrintAuthors}                {author}
     +{.} { }                            {title}
     +{.} { }                            {part}
     +{:} { }                            {subtitle}
     +{.} { \PrintContributions}         {contribution}
     +{.} { \PrintConference}            {conference}
     +{.} { \PrintBook}                  {book}
     +{.} { In \PrintEditorsA}              {editor}
     +{.} {  \emph}                         {booktitle}
     +{,} { pages~}                      {pages}
     +{,} { }                            {publisher}
     +{,} { }                            {organization}
     +{,} { }                            {address}
     +{,} { \PrintDateB}                 {date}
     +{.} { \PrintTranslation}           {translation}
     +{.} { Reprinted in \PrintReprint}  {reprint}
     +{.} { }                            {note}
     +{.} {}                             {transition}
 }

\BibSpec{conference}{%
    +{}  {}                        {title}
    +{}  {\PrintConferenceDetails} {transition}
}

\BibSpec{innerbook}{%
    +{.} { \emph}                       {title}
    +{.} { }                            {part}
    +{:} { \emph}                       {subtitle}
    +{.} { }                            {series}
    +{,} { \voltext}                    {volume}
    +{.} { Edited by \PrintNameList}    {editor}
    +{.} { Translated by \PrintNameList}{translator}
    +{.} { \PrintContributions}         {contribution}
    +{.} { }                            {publisher}
    +{.} { }                            {organization}
    +{,} { }                            {address}
    +{,} { \PrintEdition}               {edition}
    +{,} { \PrintDateB}                 {date}
    +{.} { }                            {note}
    +{.} {}                             {transition}
}

\BibSpec{report}{%
    +{}  {\PrintPrimary}                {transition}
    +{.} { \emph}                       {title}
    +{.} { }                            {part}
    +{:} { \emph}                       {subtitle}
    +{.} { \PrintContributions}         {contribution}
    +{.} { Technical Report }           {number}
    +{,} { }                            {series}
    +{.} { }                            {organization}
    +{,} { }                            {address}
    +{,} { \PrintDateB}                 {date}
    +{.} { \PrintTranslation}           {translation}
    +{.} { Reprinted in \PrintReprint}  {reprint}
    +{.} { }                            {note}
    +{.} {}                             {transition}
}

\BibSpec{thesis}{%
    +{}  {\PrintAuthors}                {author}
    +{,} { \emph}                       {title}
    +{:} { \emph}                       {subtitle}
    +{.} { \PrintThesisType}            {type}
    +{.} { }                            {organization}
    +{,} { }                            {address}
    +{,} { \PrintDateB}                 {date}
    +{.} { \PrintTranslation}           {translation}
    +{.} { Reprinted in \PrintReprint}  {reprint}
    +{.} { }                            {note}
    +{.} {}                             {transition}
}

\newtheorem{thm}{Theorem}
\newtheorem{prop}[thm]{Proposition}
\newtheorem{lem}[thm]{Lemma}
\newtheorem{cor}[thm]{Corollary}

\theoremstyle{definition}
\newtheorem{defn}[thm]{Definition}
\theoremstyle{remark}
\newtheorem{rem}[thm]{Remark}
\newtheorem{example}[thm]{Example}
\theoremstyle{definition}
\newtheorem{proc}[thm]{Procedure}

\providecommand{\Space}[3][]{\ensuremath{\mathbb{#2}^{#3}_{#1}{}}}
\providecommand{\SL}[1][2]{\ensuremath{\mathrm{SL}_{#1}(\Space{R}{})}}

\providecommand{\tr}{\mathop{\mathrm{tr}}}
\providecommand{\scalar}[3][\relax]{\left\langle #2,#3 
        \right\rangle\ifx#1\relax\else_{#1}\fi}
\providecommand{\modulus}[2][\relax]{\left| #2 \right|\ifx#1\relax\else_{#1}\fi}
  \providecommand{\Zbl}[1]{Zbl\href{http://www.emis.de:80/cgi-bin/zmen/ZMATH/en/zmathf.html?first=1&maxdocs=3&type=html&an=#1&format=complete}{#1}}
\providecommand{\rmi}{\mathrm{i}}
\providecommand{\rme}{\mathrm{e}}

\providecommand{\nscalar}[3][\relax]{\left[ #2,#3 
        \right]\ifx#1\relax\else_{#1}\fi}

\begin{document}
\title[Conformal Parametrisation of Loxodromes]
{Conformal Parametrisation of Loxodromes\\ by Triples of Circles}

\author[V.V. Kisil and J. Reid]%
{\href{http://www.maths.leeds.ac.uk/~kisilv/}{Vladimir V. Kisil} and James Reid}

\address{%
School of Mathematics\\
University of Leeds\\
Leeds LS2\,9JT\\
UK
}
\email{\href{mailto:kisilv@maths.leeds.ac.uk}{kisilv@maths.leeds.ac.uk}}
\urladdr{\url{http://www.maths.leeds.ac.uk/~kisilv/}}

\address{School of Mathematics\\ The University of Manchester\\ Manchester M13 9PL\\ UK}

\email{\href{mailto:Jamesjohnreid@googlemail.com}{Jamesjohnreid@googlemail.com}}

\date{\today}
\dedicatory{Dedicated to Prof. Wolfgang Spr\"o\ss{}ig on the occasion
  of his 70\(^{\text{th}}\) birthday}

\begin{abstract}
  We provide a parametrisation of a loxodrome by three specially
  arranged cycles. The parametrisation is covariant under fractional
  linear transformations of the complex plane and naturally encodes
  conformal properties of loxodromes. Selected geometrical examples
  illustrate the usage of parametrisation. Our work extends the set of
  objects in Lie sphere geometry---circle, lines and points---to the
  natural maximal conformally-invariant family, which also includes
  loxodromes.
\end{abstract}


\maketitle

\section{Introduction}
\label{sec:introduction}

It is easy to come across shapes of logarithmic spirals, as on
Fig.~\ref{fig:log-lox-init}(a), looking either on a sunflower, a snail
shell or a remote galaxy. It is not surprising since the fundamental
differential equation \(\dot{y}=\lambda y\),
\(\lambda\in \Space{C}{}\) serves as a first approximation to many
natural processes.
\begin{figure}[htbp]
  \centering
  (a)\includegraphics[scale=.37]{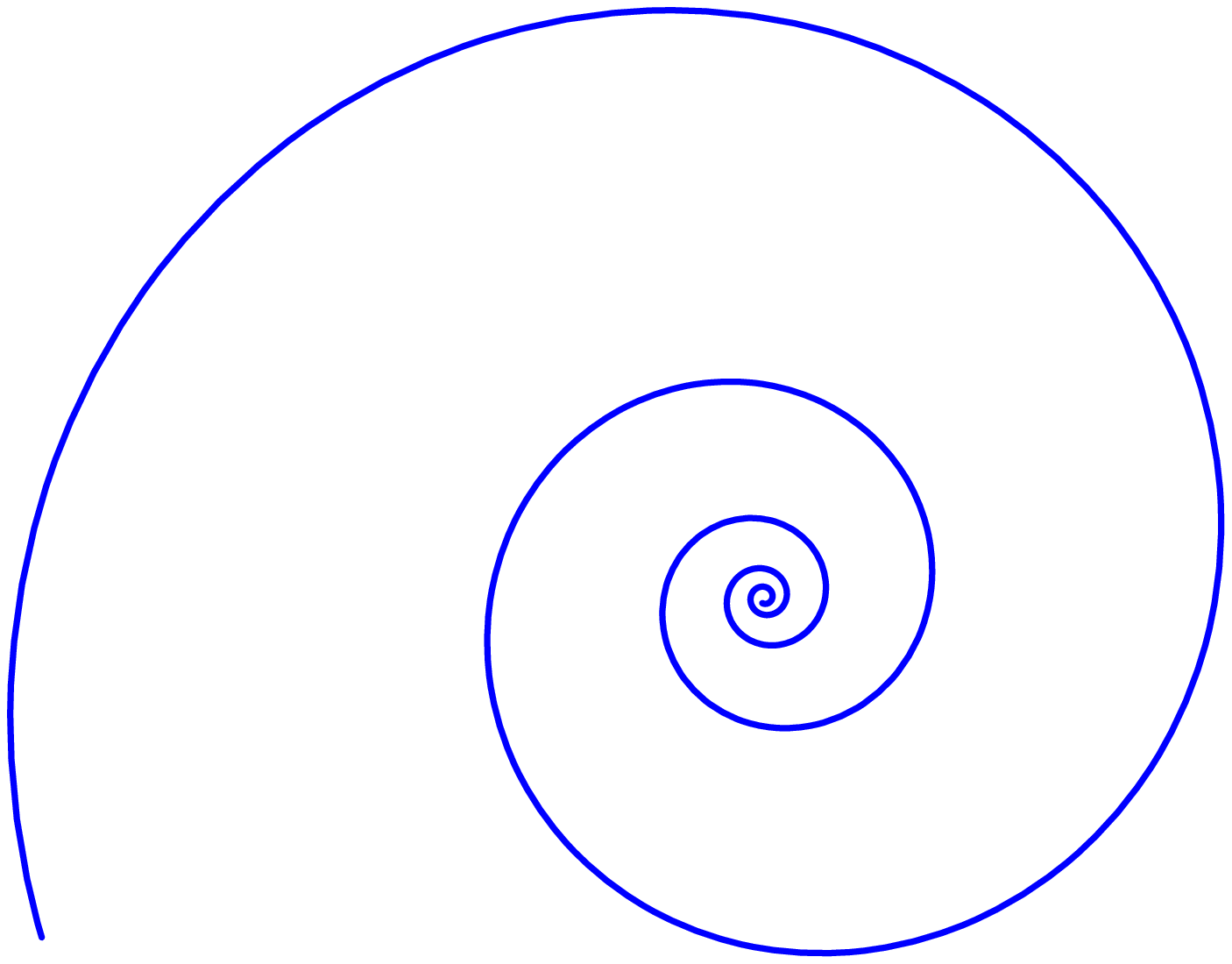}  \hfil
  (b)\includegraphics[scale=.37]{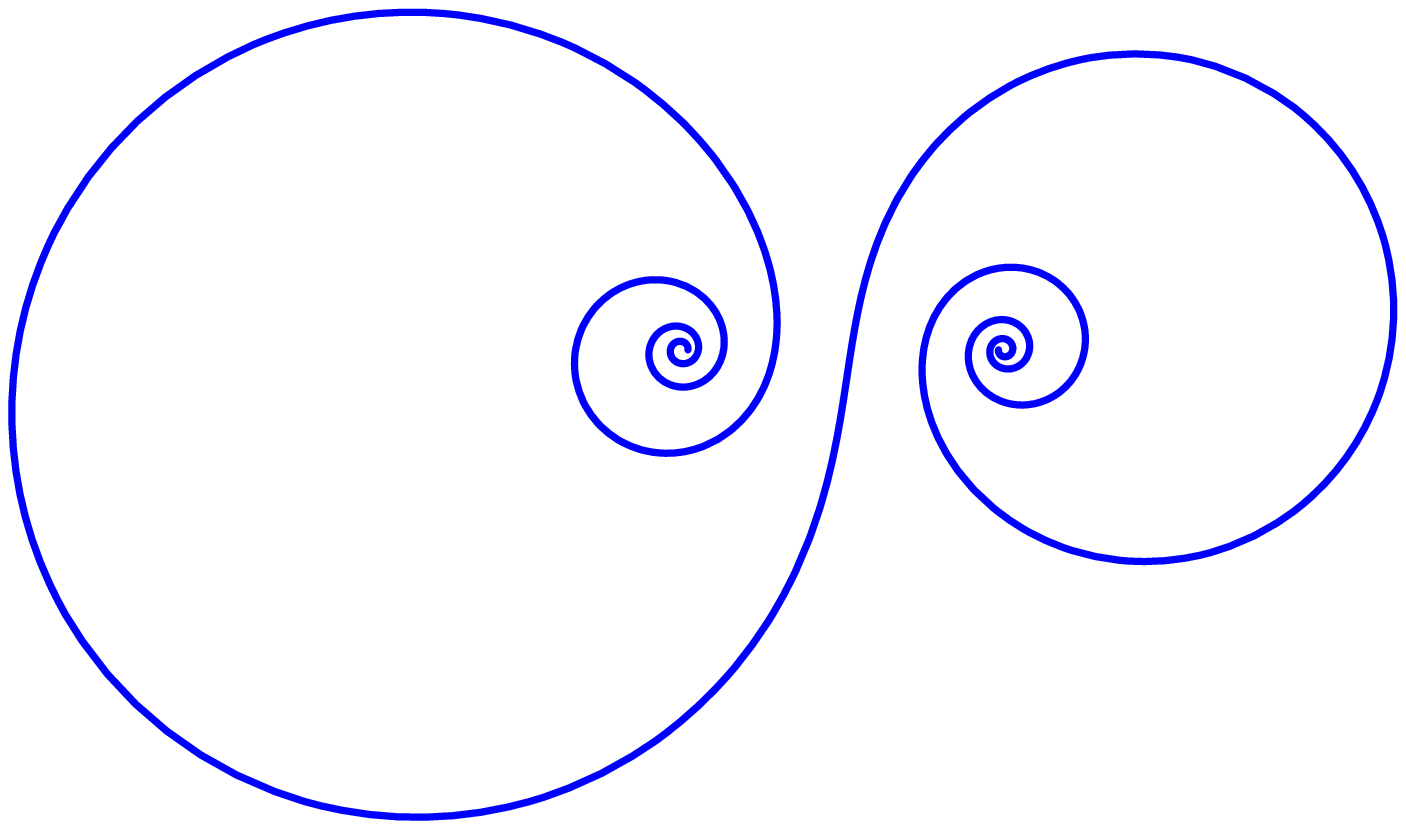}
  \caption[A logarithnic spiral and a loxodrome]{A logarithnic spiral
    (a) and its image under a fractional linear
    transformation---loxodrome (b).}
  \label{fig:log-lox-init}
\end{figure}
The main symmetries of complex analysis are build on the fractional linear
transformation (FLT):
\begin{equation}
  \label{eq:FLT-defn}
  \begin{pmatrix}
    \alpha & \beta \\
    \gamma & \delta 
  \end{pmatrix}:\ 
  z\mapsto \frac{\alpha z+\beta}{\gamma z +\delta}, \quad
    \text{ where }
    \alpha,\beta, \gamma,\delta \in \Space{C}{} \text{ and }
    \det  \begin{pmatrix}
    \alpha & \beta \\
    \gamma & \delta 
  \end{pmatrix}
  \neq 0.
\end{equation}
Thus, images of logarithmic spirals under FLT, called loxodromes, as
on Fig.~\ref{fig:log-lox-init}(b) shall not be rare.  Indeed, they
appear in many occasions from the stereographic projection of a rhumb
line in navigation to a preferred model of a Carleson arc in the
theory singular integral
operators~\cites{BoettcherKarlovich01a,%
  BishopBoettcherKarlovichSpitkovsky99a}. Furthermore,
loxodromes are orbits of one-parameter continuous groups of FLT of
loxodromic type~\citelist{\cite{Beardon95}*{\S~4.3} \cite{Simon11a}*{\S~9.2}
  \cite{Vasilevski08a}*{\S~9.2}}.

This setup motivates a search for effective tools to deal with
FLT-invariant properties of loxodromes. They were studied from
a differential geometry point of view in several recent paper
\cites{Bolt07a,Porter93a,Porter98a,Porter98b,Porter07a}. In this work
we develop a ``global'' description which matches the Lie sphere
geometry framework, see Rem.~\ref{re:historic}.

The outline of the paper is as follows. After preliminaries on FLT and
invariant geometry of cycles (Section~\ref{sec:preliminaries}) we review
the basics of logarithmic spirals and loxodromes
(Section~\ref{sec:fract-line-transf}). A new parametrisation of
loxodromes is introduced in Section~\ref{sec:loxodr-encod-as} and several
examples illustrate its usage in Section~\ref{sec:applications}.
Section~\ref{sec:open-questions} frames
our work within a wider approach~\cites{Kisil15a,Kisil14b},
which extends Lie sphere geometry.
A brief list of open questions concludes the paper.

 \section{Preliminaries: Fractional Linear Transformations  and Cycles}
\label{sec:preliminaries}

In this section we provide some necessary background in Lie geometry
of circles, fractional-linear transformations and
Fillmore--Springer--Cnops construction (FSCc). Regretfully, the latter
remains largely unknown in the context of complex numbers despite of
its numerous advantages. We will have some further discussion of this
in Rem.~\ref{re:historic} below.

The right way~\cite{Simon11a}*{\S~9.2} to think about
FLT~\eqref{eq:FLT-defn} is through the \emph{projective complex line}
\(P\Space{C}{}\). It is the family of cosets in \(\Space{C}{2}\) with
respect to the equivalence relation \(
\begin{pmatrix}
  w_1\\w_2
\end{pmatrix}
\sim
\begin{pmatrix}
  \alpha w_1\\ \alpha w_2
\end{pmatrix}\) for all nonzero \(\alpha\in\Space{C}{}\). Conveniently
\(\Space{C}{}\) is identified with a part of \(P\Space{C}{}\) by
assigning the coset of \(
\begin{pmatrix}
  z\\1
\end{pmatrix}\) to \(z\in\Space{C}{}\). Loosely speaking
\(P\Space{C}{}=\Space{C}{}\cup \{\infty\}\), where \(\infty\) is the
coset of \(
\begin{pmatrix}
  1\\0
\end{pmatrix}\). The pair \([w_1:w_2]\) with
\(w_2\neq 0\) gives \emph{homogeneous coordinates} for
\(z=w_1/w_2 \in\Space{C}{}\). Then, the linear map
\(\Space{C}{2}\rightarrow \Space{C}{2}\)
\begin{equation}
  \label{eq:gl2-linear-map}
  M:\ \begin{pmatrix}
    w_1\\
    w_2
  \end{pmatrix}
  \mapsto
  \begin{pmatrix}
    w_1'\\
    w_2'
  \end{pmatrix}
  =
    \begin{pmatrix}
    \alpha w_1 + \beta w_2\\
    \gamma w_1 + \delta w_2
  \end{pmatrix},
  \text{ where }
    M=\begin{pmatrix}
    \alpha & \beta \\
    \gamma & \delta 
  \end{pmatrix}
  \in\mathrm{GL}_2(\Space{C}{})
\end{equation}
factors from \(\Space{C}{2}\) to \(P\Space{C}{}\) and coincides
with~\eqref{eq:FLT-defn} on \(\Space{C}{}\subset P\Space{C}{}\).

Generic equations of cycle in real and complex coordinates \(z=x+\rmi
y\) are:
\begin{equation}
  \label{eq:cycle-def-real}
  k(x^2+y^2)-2lx-2ny+m=0\qquad \text{or} \qquad
  k z\bar{z}-\bar{L}z-{L}\bar{z} +m=0\,,
\end{equation}
where \((k,l,n,m)\in\Space{R}{4}\) and \(L=l+\rmi n\).  This includes
lines (if \(k=0\)), points as zero-radius circles (if \(l^2+n^2-mk=0\)) and
proper circles otherwise. Homogeneity of~\eqref{eq:cycle-def-real}
suggests that \((k,l,m,n)\) shall be considered as homogeneous
coordinates \([k:l:m:n]\) of a point in three-dimensional projective
space \(P\Space{R}{3}\).

The homogeneous form of cycle's equation~\eqref{eq:cycle-def-real} for
\(z=[w_1:w_2]\) can be written\footnote{Of course, this is not the
  only possible presentation. However, this form is particularly
  suitable to demonstrate FLT-invariance~\eqref{eq:product-invariance}
  of the cycle product below.}  using matrices as follows:
\begin{equation}
  \label{eq:quadratic-matrices}
  k w_1\bar{w}_1-\bar{L}w_1\bar{w}_2-{L}\bar{w}_1 w_2+mw_2\bar{w}_2=
  \begin{pmatrix}
    -\bar{w}_2&\bar{w}_1
  \end{pmatrix}
  \begin{pmatrix}
    \bar{L}&-m\\
    k&-{L}
  \end{pmatrix}
  \begin{pmatrix}
    {w}_1\\w_2
  \end{pmatrix}=0.
\end{equation}
From now on we identify a cycle \(C\) given by~\eqref{eq:cycle-def-real}
with its \(2\times 2\) matrix \(  \begin{pmatrix}
    \bar{L}&-m\\
    k&-{L}
  \end{pmatrix}
  \), this is called the \emph{Fillmore--Springer--Cnops construction}
  (FSCc) . Again, \(C\) shall be treated up to the equivalence
  relation \(C \sim tC\) for all real \(t\neq 0\).  Then, the linear
  action~\eqref{eq:gl2-linear-map} corresponds to some action on
  \(2\times 2\) cycle matrices by the intertwining identity:
\begin{equation}
  \label{eq:intertwine-vect-matrix}
  \begin{pmatrix}
    -\bar{w}'_2&\bar{w}'_1
  \end{pmatrix}
  \begin{pmatrix}
    \bar{L}'&-m'\\
    k'&-{L}'
  \end{pmatrix}
  \begin{pmatrix}
    {w}'_1\\w'_2
  \end{pmatrix}=
  \begin{pmatrix}
    -\bar{w}_2&\bar{w}_1
  \end{pmatrix}
  \begin{pmatrix}
    \bar{L}&-m\\
    k&-{L}
  \end{pmatrix}
  \begin{pmatrix}
    {w}_1\\w_2
  \end{pmatrix}.
\end{equation}
 Explicitly, for \(M\in\mathrm{GL}_2(\Space{C}{})\) those actions are:
 \begin{equation}
   \label{eq:cycle-matrice-transformation}
  \begin{pmatrix}
    {w}'_1\\w'_2
  \end{pmatrix}=
  M
  \begin{pmatrix}
    {w}_1\\w_2
  \end{pmatrix},
  \quad\text{and}\quad
  \begin{pmatrix}
    \bar{L}'&-m'\\
    k'&-{L}'
  \end{pmatrix}=
  \bar{M}
  \begin{pmatrix}
    \bar{L}&-m\\
    k&-{L}
  \end{pmatrix}M^{-1}\,,
\end{equation}
where \(\bar{M}\) is the component-wise complex conjugation of
\(M\). Note, that FLT \(M\)~\eqref{eq:FLT-defn} corresponds to a
linear transformation \(C\mapsto M(C):=\bar{M}CM^{-1}\) of cycle matrices
in~\eqref{eq:cycle-matrice-transformation}. A quick calculation shows
that \(M(C)\) indeed has real off-diagonal elements as
required for a FSCc matrix.

This paper essentially depends on the following
\begin{prop}
  Define a \emph{cycle product} of two cycles \(C\) and \(C'\) by:
  \begin{equation}
    \label{eq:cycle-product-defn}
    \scalar{C}{C'}:=\tr (C\bar{C}')=L\bar{L}'+\bar{L}L'-mk'-km'.
  \end{equation}
  Then, the cycle product is FLT-invariant:
  \begin{equation}
    \label{eq:product-invariance}
    \scalar{M(C)}{M(C')} = \scalar{C}{C'}\qquad  \text{for any }M\in
  \mathrm{SL}_2(\Space{C}{})\,.
  \end{equation}
\end{prop}
\begin{proof}
  Indeed:
  \begin{align*}
    \scalar{M(C)}{M(C')} &= \tr(M(C)\overline{M(C')})\\
    &=\tr(\bar{M}CM^{-1} M\bar{C}'\bar{M}^{-1})\\
    &=\tr(\bar{M}C\bar{C}'\bar{M}^{-1})\\
    &=\tr(C\bar{C}')\\
    &=\scalar{C}{C'},
  \end{align*}
  using the invariance of trace.
\end{proof}
Note that the cycle product~\eqref{eq:cycle-product-defn} is
\emph{not} positive definite, it produces a Lorentz-type metric in
\(\Space{R}{4}\). Here are some relevant examples of geometric
properties expressed through the cycle product:
\begin{example}
  \label{ex:cycle-product-geometric}
  \begin{enumerate}
  \item If \(k=1\) (and \(C\) is a proper circle), then \(\scalar{C}{C}\!/2\)
    is equal to the square of radius of \(C\). In particular
    \(\scalar{C}{C}=0\) indicates a zero-radius circle representing a
    point.
  \item If \(\scalar{C_1}{C_2}=0\) for non-zero radius cycles \(C_1\)
    and \(C_2\), then they intersects at the right angle.
  \item If \(\scalar{C_1}{C_2}=0\) and \(C_2\) is zero-radius circle, then
    \(C_1\) passes the point represented by \(C_2\). 
  \end{enumerate}
\end{example}
In general, a combination
of~\eqref{eq:cycle-matrice-transformation}
and~\eqref{eq:product-invariance} yields that a consideration of FLT
in \(\Space{C}{}\) can be replaced by linear algebra in the space of
cycles \(\Space{R}{4}\) (or rather \(P\Space{R}{3}\)) with an indefinite
metric, see~\cite{GohbergLancasterRodman05a} for the latter.

\begin{figure}[htbp]
  \centering
  \includegraphics[scale=.9]{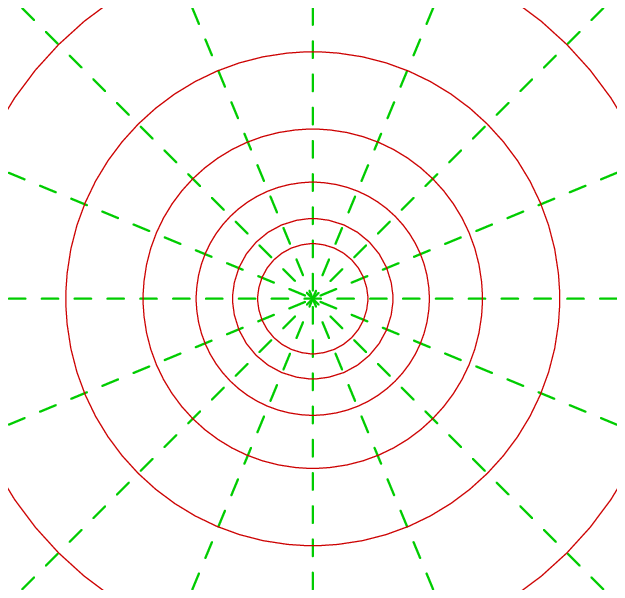}\quad
  \includegraphics[scale=1]{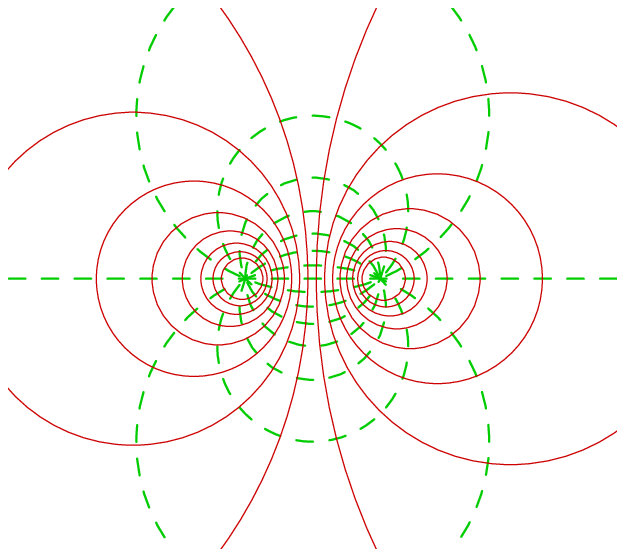}
  \caption[Elliptic and hyperbolic pencils]{Orthogonal elliptic
    (green-dashed) and hyperbolic (red-solid) pencils of cycles. Left
    drawing shows the standard case and the right---generic, which is
    the image of the standard pencils under FLT.}
  \label{fig:ellipt-hyperb-penc}
\end{figure}

A spectacular (and needed later) illustration of this approach is
orthogonal pencils of cycles. Consider a collection of all cycles
passing two different points in \(\Space{C}{}\), it is called an
\emph{elliptic pencil}. A beautiful and non-elementary fact of the
Euclidean geometry is that cycles orthogonal to every cycle in the
elliptic pencil fill the entire plane and are disjoint, the family is
called a \emph{hyperbolic pencil}. The statement is obvious in the
standard arrangement when the elliptic pencil is formed by straight
lines---cycles passing the origin and infinity. Then, the hyperbolic
pencil consists of the concentric circles, see
Fig.~\ref{fig:ellipt-hyperb-penc}. For the sake of completeness, a
\emph{parabolic pencil} (not used in this paper) formed by all circles
touching a given line at a given point,
\cite{Kisil12a}*{Ex.~6.10} contains further extensions and illustrations.
See~\cite{Vasilevski08a}*{\S~11.8} for an example of cycle pencils'
appearance in operator theory.

This picture trivialises a bit in the language of cycles. A pencil of
cycles (of any type!) is a linear span \(t C_1+(1-t)C_2\) of two
arbitrary different cycles \(C_1\) and \(C_2\) from the pencil. Again,
this is easier to check for the standard pencils. A pencil is
elliptic, parabolic or hyperbolic depending on which inequality
holds~\cite{Kisil12a}*{Ex.~5.28.ii}:
\begin{equation}
  \label{eq:cauchy-schwartz-cycles}
  \scalar{C_1}{C_2}^2 
  \lesseqqgtr
  \scalar{C_1}{C_1}
  \scalar{C_2}{C_2}.
\end{equation}

Then, the orthogonality of cycles on the plane is exactly their
orthogonality as vectors with respect to the indefinite
cycle product~\eqref{eq:cycle-product-defn}. For cycles in the standard
pencils this is immediately seen from the explicit expression of the
product  \(\scalar{C}{C'}=L\bar{L}'+\bar{L}L'-mk'-km'\) in
cycle components. Finally,
linearization~\eqref{eq:cycle-matrice-transformation} of FLT in the
cycle space shows that a pencil (i.e. a linear span) is transformed
to a pencil and FLT-invariance~\eqref{eq:product-invariance} of the cycle
product guarantees that the orthogonality of two pencils is
preserved.

\begin{rem}
  \label{re:historic}
  A sketchy historic overview (we apologise for any important
  omission!)  starts from the concept of Lie sphere geometry,
  see~\cite{Benz07a}*{Ch.~3} for a detailed presentation. It unifies
  circles, lines and points, which all are called cycles in this
  context (analytically it is already
  in~\eqref{eq:cycle-def-real}). The main invariant property of Lie
  sphere geometry is \emph{tangential contact}. The first radical
  advance came from the observation that cycles (through their
  parameters in~\eqref{eq:cycle-def-real}) naturally form a linear or
  projective space, see~\citelist{\cite{Pedoe95a}
    \cite{Schwerdtfeger79a}*{Ch.~1}}. The second crucial step is the
  recognition that the cycle space carries out the FLT-invariant
  indefinite metric~\citelist{\cite{Benz07a}*{Ch.~3}
    \cite{Kirillov06}*{\S~F.4}}. At the same time some presentations
  of cycles by \(2\times 2\) matrices were
  used~\citelist{\cite{Simon11a}*{\S~9.2}
    \cite{Schwerdtfeger79a}*{Ch.~1} \cite{Kirillov06}*{\S~F.4}}. Their
  main feature is that FLT in \(\Space{C}{}\) corresponds to a some
  sort of linear transform by matrix conjugation in the cycle
  space. However, the metric in the cycle space was not expressed in
  terms of those matrices.

  All three ingredients---matrix presentation with linear structure
  and the invariant product--came happily together as
  Fillmore--Springer--Cnops construction (FSCc) in the context of
  Clifford algebras~\citelist{\cite{Cnops02a}*{Ch.~4}
    \cite{FillmoreSpringer90a}}. Regretfully, FSCc have not yet
  propagated back to the most fundamental case of complex numbers,
  cf.~\cite{Simon11a}*{\S~9.2} or somewhat cumbersome techniques used
  in~\cite{Benz07a}*{Ch.~3}. Interestingly, even the founding fathers
  were not always strict followers of their own techniques,
  see~\cite{FillmoreSpringer00a}.

  A combination of all three components of Lie cycle geometry within
  FSCc facilitates further development. It was discovered that for the
  smaller group \(\SL\) there exist more types---elliptic, parabolic
  and hyperbolic--of invariant metrics in the cycle 
  space~\citelist{\cite{Kisil06a} \cite{Kisil05a} \cite{Kisil12a}*{Ch.~5}}. Based on
  the earlier work~\cite{Kirillov06}, the key
  concept of Lie sphere geometry---\emph{tangency of two cycles} \(C_1\) and
  \(C_2\)--- was expressed through the cycle
  product~\eqref{eq:cycle-product-defn} as \cite{Kisil12a}*{Ex.~5.26.ii}:
  \begin{displaymath}
    \scalar{C_1+C_2}{C_1+C_2}=0\, \quad \text{for }C_1, C_2\text{
      normalised such
      that } \scalar{C_1}{C_1}=\scalar{C_2}{C_2}=1.
  \end{displaymath}
  (Furthermore, \(C_1+C_2\) is the zero-radius cycle representing the
  point of contact.)  FSCc is useful in consideration of the Poincar\'e
  extension of M\"obius maps~\cite{Kisil15a} and continued
  fractions~\cite{Kisil14a}. In theoretical physics FSCc nicely
  describes conformal compactifications of various space-time
  models~\citelist{\cite{HerranzSantander02b} \cite{Kisil06b}
    \cite{Kisil12a}*{\S~8.1}}. Last but not least, FSCc is behind the
  Computer Algebra System (CAS) operating in Lie sphere
  geometry~\citelist{\cite{Kisil05b} \cite{Kisil14b}}. FSCc equally
  well covers not only the field of complex numbers but rings of dual
  and double numbers as well~\cite{Kisil12a}. New usage of FSCc will
  be given in the following sections in applications to loxodromes.
\end{rem}

\section{Fractional Linear Transformations and Loxodromes}
\label{sec:fract-line-transf}

In aiming for a covariant description of loxodromes we start from the
following definition.
\begin{defn}
  \label{de;standard-log-spiral}
  A \emph{standard logarithmic spiral} (SLS) with parameter
  \(\lambda\in \Space{C}{}\) is the orbit of the point 
  \(1\) under the (disconnected) one-parameter subgroup of FLT of diagonal
  matrices
  \begin{equation}
    \label{eq:diagonal-subgroup}
    D_\lambda (t)=
  \begin{pmatrix}
    \pm \rme^{\lambda t/2} &0\\
    0 &\rme^{-\lambda t/2}
  \end{pmatrix}, \qquad t\in\Space{R}{}.
  \end{equation}
\end{defn}
\begin{rem}
  Our SLS is a \emph{union} of two branches, each of them is a
  logarithmic spiral in the common sense. The three-cycle
  parametrisation of loxodromes presented below will becomes less
  elegant if those two branches need to be separated. Yet, we draw
  just one ``positive'' branch on
  Fig.~\ref{fig:log-spir-loxodr-pencils} to make it more transparent.
\end{rem}

SLS is the solution of the differential equation \(z'=\lambda z\) with
the initial value \(z(0)=\pm 1\) and has the parametric equation
\(z(t)=\pm \rme^{\lambda t}\). Furthermore, we obtain the same orbit
for \(\lambda_1\) and \(\lambda_2\in \Space{C}{}\) if
\(\lambda_1=a\lambda_2\) for real \(a\neq 0\) through a
re-parametrisation of the time \(t\mapsto a t\). Thus, SLS is
identified by the point \([\Re (\lambda) : \Im (\lambda)]\) of the
real projective line \(P\Space{R}{}\).  Thereafter the following
classification is useful:
\begin{defn}
  \label{th:sls-classification}
  SLS is 
  \begin{itemize}
  \item \emph{positive}, if \(\Re(\lambda)\cdot \Im(\lambda) > 0\);
  \item \emph{degenerate}, if \(\Re(\lambda)\cdot \Im(\lambda) = 0\);
  \item \emph{negative}, if \(\Re(\lambda)\cdot \Im(\lambda) < 0\).
  \end{itemize}
\end{defn}
Informally: a positive SLS unwinds counterclockwise, a
negative---clockwise. Degenerate SLS is the unit
circle if \(\Im(\lambda)\neq 0\) and the punctured real axis
\(\Space{R}{}\setminus \{0\}\)
if \(\Re(\lambda)\neq 0\). If \( \Re(\lambda)=\Im(\lambda)=0\) then
SLS  is the single point \(1\). 
\begin{defn}
  \label{de:loxodromes}
  A \emph{logarithmic spiral} is the image of a SLS under a complex
  affine transformation \(z\mapsto \alpha z +\beta\), with \(\alpha\),
  \(\beta\in \Space{C}{}\).  A \emph{loxodrome} is an image of a SLS
  under a generic FLT~\eqref{eq:FLT-defn}.
\end{defn}
Obviously, a complex affine transformation is FLT with the upper
triangular matrix \(
\begin{pmatrix}
  \alpha & \beta\\0 &1
\end{pmatrix}
\). Thus, logarithmic spirals form an affine-invariant (but not
FLT-invariant) subset of loxodromes.  Thereafter, loxodromes (and their
degenerate forms---circles, straight lines and points) extend the
notion of cycles from the Lie sphere geometry,
cf. Rem.~\ref{re:historic}.

By the nature of Defn.~\ref{de:loxodromes}, the parameter \(\lambda\)
and the corresponding classification from
Defn.~\ref{th:sls-classification} remain meaningful for logarithmic
spirals and loxodromes. FLTs eliminate distinctions between circles and
straight lines, but for degenerate loxodromes
(\(\Re(\lambda)\cdot \Im(\lambda) = 0\)) we still can note the difference
between two cases of \(\Re(\lambda)\neq 0\) and
\(\Im(\lambda)\neq 0\): orbits of former are whole circles (straight
lines) while latter orbits are only arcs of circles (segments of
lines).

The immediate consequence of Defn.~\ref{de:loxodromes} is 
\begin{prop}
  The collection of all loxodromes is a FLT-invariant family. 
  Degenerate loxodromes---(arcs of) circles and (segments) of straight
  lines---form a FLT-invariant subset of loxodromes.
\end{prop}

As mentioned above, SLS is completely characterised by the point
\([\Re (\lambda) : \Im (\lambda)]\) of the real projective line
\(P\Space{R}{}\) extended by the additional point
\([0 : 0]\)\footnote{Pedantic consideration of the trivial case
  \(\Re (\lambda) = \Im (\lambda)=0\) will be often omitted in the
  following discussion.}.  In the standard way,
\([\Re (\lambda) : \Im (\lambda)]\) is associated with the real value
\(\tilde{\lambda}:=2\pi \Re (\lambda)/\Im (\lambda)\) extended by
\(\infty\) for \(\Im (\lambda)=0\) and symbol \(\frac{0}{0}\) for
the \(\Re(\lambda)=\Im(\lambda)=0\) cases. Geometrically,
\(a=\exp(\tilde{\lambda})\in\Space[+]{R}{}\) represents the next point
after \(1\), where the given SLS branch meets the real positive
half-axis after one full counterclockwise turn. Obviously, \(a>1\) and
\(a<1\) for positive and negative SLS, respectively.  For a degenerate
SLS:
\begin{enumerate}
\item with \(\Im(\lambda)\neq 0\) we obtain \(\tilde{\lambda}=0\) and \(a=1\);
\item with \(\Re(\lambda)\neq 0\) we consistently \emph{define} \(a=\infty\). 
\end{enumerate}

In essence, a loxodrome \(\Lambda\) is defined by the pair
\((\tilde{\lambda}, M)\), where \(M\) is the FLT mapping \(\Lambda\) to SLS
with the parameter \(\tilde{\lambda}\). While \(\tilde{\lambda}\) is completely
determined by \(\Lambda\), a map \(M\) is not.
\begin{prop}
  \begin{enumerate}
  \item The subgroup of FLT which maps SLS with the parameter
    \(\tilde{\lambda}\) to itself consists of products
    \(D_{\tilde{\lambda}}(t) R^{\varepsilon}\), \(\varepsilon=0,1\) of
    transformations \(D_{\tilde{\lambda}}(t)=D_\lambda(t)\), \(\lambda=\tilde{\lambda}+2\pi\rmi\)~\eqref{eq:diagonal-subgroup} and
    branch-swapping reflections:
    \begin{equation}
      \label{eq:reflections}
      R=    \begin{pmatrix}
        0&-1\\1&0
      \end{pmatrix}: z\mapsto - z^{-1}\,.
    \end{equation}
  \item Pairs \((\tilde{\lambda},M)\) and \((\tilde{\lambda}',M')\) define the same
    loxodrome if and only if
    \begin{enumerate}
    \item \(\tilde{\lambda} = \tilde{\lambda} '\);
    \item \(M=D_{\tilde{\lambda}}(t) R^{\varepsilon} M'\) for \(\varepsilon=0,1\)
      and \(t\in\Space{R}{}\).
    \end{enumerate}
  \end{enumerate}
\end{prop}
\begin{rem}
  Often loxodromes appear as orbits of one-parameter continuous
  subgroup of loxodromic FLT, which are characterised by a non-real
  trace~\citelist{\cite{Beardon95}*{\S~4.3} \cite{Simon11a}*{\S~9.2}
    \cite{Vasilevski08a}*{\S~9.2}}. In the above notations such a
  subgroup is \(MD_{\tilde{\lambda}}(t)M^{-1}\), thus the common
  presentation is not much different from the above
  \((\tilde{\lambda}, M)\)-parametrisation.
\end{rem}

\section{Three-cycle Parametrisation of Loxodromes}
\label{sec:loxodr-encod-as}

Although pairs \((\tilde{\lambda}, M)\) provide a parametrisation of
loxodromes, the following alternative is more operational. It is
inspired by the orthogonal pairs of elliptic and hyperbolic pencils
described in Section~\ref{sec:preliminaries}.
\begin{defn}
  A \emph{three-cycle parametrisation} \(\{C_1,C_2,C_3\}\) of a
  non-degenerate SLS \(\tilde{\lambda}\) satisfies the following conditions:
  \begin{enumerate}
  \item \(C_1\) is the straight line passing the origin;
  \item \(C_2\) and \(C_3\)  are two circles  with their
    centres at the origin;
  \item \(\Lambda\) passes the intersection points \(C_1\cap C_2\) and
    \(C_1\cap C_3\); and
  \item A branch of  \(\Lambda\) makes one full
    counterclockwise turn between  intersection points \(C_1\cap C_2\) and
    \(C_1\cap C_3\) belonging to a ray in \(C_1\) from the origin. 
  \end{enumerate}
\end{defn}
We say that three-cycle parametrisation is \emph{standard} if \(C_1\)
is the real axis and \(C_2\) is the unit circle, then
\(C_3=\{z: \modulus{z}=\exp(\tilde{\lambda})\}\).  A three-cycle
parametrisation can be consistently extended to a degenerate SLS
\(\Lambda\) as follows:
\begin{itemize}
\item[\(\tilde{\lambda}=0\):] any straight line \(C_1\) passing the origin and the unit circles
  \(C_2=C_3=\Lambda\);
\item[\(\tilde{\lambda}=\infty\):] the real axis as \(C_1=\Lambda\), the unit circle as \(C_2\)
  and \(C_3=(0,0,0,1)\) being the zero-radius circle at infinity. 
\end{itemize}

Since cycles are elements of the projective space, the following
\emph{normalised cycle product}:
\begin{equation}
  \label{eq:norm-cycle-prod}
  \nscalar{C_1}{C_2}:=\frac{\scalar{C_1}{C_2}}{\sqrt{\scalar{C_1}{C_1}
          \scalar{C_2}{C_2}}}
\end{equation}
is more meaningful than the cycle
product~\eqref{eq:cycle-product-defn} itself. Note that,
\(\nscalar{C_1}{C_2}\) is defined only if neither \(C_1\) nor \(C_2\)
is a zero-radius cycle (i.e. a point). Also, the normalised cycle
product is \(\mathrm{GL}_2(\Space{C}{})\)-invariant in comparison to
\(\mathrm{SL}_2(\Space{C}{})\)-invariance
in~\eqref{eq:product-invariance}.

A reader will instantly recognise the familiar pattern of the cosine
of angle between two vectors appeared
in~\eqref{eq:norm-cycle-prod}. Simple calculations show that this
geometric interpretation 
is very explicit in two special cases of our interest.
\begin{lem}
  \label{le:norm-prod-geometr}
  \begin{enumerate}
    \item Let \(C_1\) and \(C_2\) be two straight lines passing the
    origin with slopes \(\tan\phi_1\) and \(\tan\phi_2\)
    respectively. Then \(C_2= D_{x+\rmi y}(1) C_1\) for
      transformation~\eqref{eq:diagonal-subgroup} with any
      \(x\in\Space{R}{}\) and \(y=\phi_2-\phi_1\) satisfying the relations:
    \begin{equation}
      \label{eq:inner-product-lines}
      \nscalar{C_1}{C_2}=\cos y\,.
    \end{equation}
  \item Let \(C_1\) and \(C_2\) be two circles centred at the origin
    and radii \(r_1\) and \(r_2\) respectively. Then
    \(C_2= D_{x+\rmi y}(1) C_1\) for
    transformation~\eqref{eq:diagonal-subgroup} with any
    \(y\in\Space{R}{}\) and \(x=\log(r_2) -\log(r_1)\) satisfying the
    relations:
    \begin{equation}
      \label{eq:inner-product-concentric}
      \nscalar{C_1}{C_2}
      =\cosh x\,.
      \end{equation}
  \end{enumerate}
\end{lem}
Note the explicit elliptic-hyperbolic analogy
between~\eqref{eq:inner-product-lines}
and~\eqref{eq:inner-product-concentric}. By the way, both expressions
produce real \(x\) and \(y\) due to
inequality~\eqref{eq:cauchy-schwartz-cycles} for the respective types
of pencils. Now we can deduce the following properties of three-cycle
parametrisation.
\begin{prop}
  For a given SLS \(\Lambda\) with a
    parameter \(\lambda\):
  \begin{enumerate}
  \item Any transformation~\eqref{eq:diagonal-subgroup} maps a
    three-cycle parametrisation of \(\Lambda\) to another three-cycle
    parametrisation of \(\Lambda\).
  \item For any two three-cycle parametrisations \(\{C_1,C_2,C_3\}\)
    and \(\{C_1',C_2',C_3'\}\) there exists \(t_0\in\Space{R}{}\) such
    that \(C_j'=D_\lambda(t_0)C_j\) for
    \(D_\lambda(t_0)\)~\eqref{eq:diagonal-subgroup} and
    \(j=1, 2,3\).
  \item The parameter \(\tilde{\lambda}=2\pi\Re(\lambda)/\Im(\lambda)\) of
    SLS can be recovered form its three-cycle parametrisation by the
    relation:
    \begin{equation}
      \label{eq:lambda-from-triple}
      \tilde{\lambda}= 
      \arccosh\nscalar{C_2}{C_3}
      \quad \text{and} \quad
      \lambda \sim \tilde{\lambda}+2\pi\rmi\,.
    \end{equation}
  \end{enumerate}
\end{prop}
\begin{proof}
  The first statement is obvious.  For the second we take
  \(D_\lambda(t_0): \Lambda \rightarrow \Lambda \) which maps
  \(C_1\cap C_2\) to \(C_1'\cap C_2'\), this transformation 
  maps \(C_j\mapsto C_j'\) for \(j=1,2,3\). Finally, the last statement
  follows from~\eqref{eq:inner-product-concentric}. 
\end{proof}

Note that expression~\eqref{eq:lambda-from-triple} is
FLT-invariant. Since any loxodrome is an image of SLS under FLT we
obtain a three-cycle parametrisation of loxodromes as follows.
\begin{prop}
  \label{pr:loxodrome-triple}
  \begin{enumerate}
  \item Any three-cycle parametrisation \(\{C_1,C_2,C_3\}\) of SLS
    has the following FLT-invariant properties:
    \begin{enumerate}
    \item \label{it:orthogonal}
      \(C_1\) is orthogonal to \(C_2\) and \(C_3\);
    \item \label{it:disjoint}
      \(C_2\) and \(C_3\) either disjoint or coincide\footnote{Recall
        that if \(C_2=C_3\), then SLS is degenerate and coincide with \(C_2=C_3\).}.
    \end{enumerate}
  \item For any FLT \(M\) and a three-cycle 
    parametrisation  \(\{C_1',C_2',C_3'\}\) of SLS, three cycles \(C_j=M
    (C_j')\),  \(j=1,2,3\) satisfy the above conditions
    \eqref{it:orthogonal} and \eqref{it:disjoint}.
  \item \label{item:FLT-existance} For any triple of cycles
    \(\{C_1,C_2,C_3\}\) satisfying the above conditions
    \eqref{it:orthogonal} and \eqref{it:disjoint} there exist a
    FLT \(M\) such cycles \(\{M(C_1), M(C_2), M(C_3)\}\)
    provide a three-cycle parametrisation of SLS with the parameter
    \(\tilde{\lambda}\)~\eqref{eq:lambda-from-triple}. FLT \(M\) is
    uniquely defined by the additional condition that   \(\{M(C_1),
    M(C_2), M(C_3)\}\) is a \emph{standard} parametrisation of SLS.
  \end{enumerate}
\end{prop}
\begin{proof}
  The first statement is obvious, the second follows because
  properties     \eqref{it:orthogonal} and \eqref{it:disjoint} are
  FLT-invariant.

  For \eqref{item:FLT-existance} in the degenerate case \(C_2=C_3\):
  any \(M\) which sends \(C_2=C_3\) to the unit circle will do the
  job.  If \(C_2\neq C_3\) we explicitly describe below the procedure, which
  produces FLT \(M\) mapping the loxodrome to SLS.
\end{proof}

  \begin{proc}
    \label{proc:build-normalisation}
    Two disjoint cycles \(C_2\) and \(C_3\) span a hyperbolic pencil
    \(H\) as described in Section~\ref{sec:preliminaries}. Then
    \(C_1\) belongs to the elliptic \(E\) pencil orthogonal to
    \(H\). Let \(C_0\) and \(C_\infty\) be the two zero-radius cycles
    (points) from the hyperbolic pencil \(H\).  Every cycle in \(E\),
    including \(C_1\), passes \(C_0\) and \(C_\infty\), we label those two
    in such a way that
    \begin{itemize}
    \item for a positive \(\tilde{\lambda}\) cycle \(C_3\) is between
      \(C_2\) and \(C_\infty\); and
    \item for a negative \(\tilde{\lambda}\) cycle \(C_3\) is between
      \(C_2\) and \(C_0\).
    \end{itemize}
    Here ``between'' for cycles means ``between'' for their
    intersection points with \(C_1\).  Finally, let \(C_u\) be any of
    two intersection points \(C_1\cap C_2\). Then, there exists
    the unique FLT \(M\) such that \(M: C_0\mapsto 0\), \(M: C_u\mapsto 1\)
    and \(M: C_\infty \mapsto \infty\). We will call \(M\) the
    \emph{standard FLT associated} to the three-cycle parametrisation
    \(\{C_1,C_2,C_3\}\) of the loxodrome.
  \end{proc}

\begin{rem}
  To complement the construction of the standard FLT \(M\) associated
  to the three-cycle parametrisation \(\{C_1,C_2,C_3\}\) from
  Procedure~\ref{proc:build-normalisation}, we can describe the inverse
  operation. For the loxodrome, which is the image of SLS with the
  parameter \(\lambda\) under FLT \(M\), we define the \emph{standard
    three-cycle parametrisation}
  \(\{M (\Space{R}{}), M(C_u), M(C_\lambda) \}\) as the image of the
  standard parametrisation of the SLS under \(M\).  Here \(\Space{R}{}\)
  is the real axis, \(C_u=\{z: \modulus{z}=1\}\) is the unit circle
  and \(C_\lambda=\{z: \modulus{z}=\exp(\tilde{\lambda})\}\).
\end{rem}
\begin{figure}[htbp]
  \centering
  \includegraphics[scale=.7]{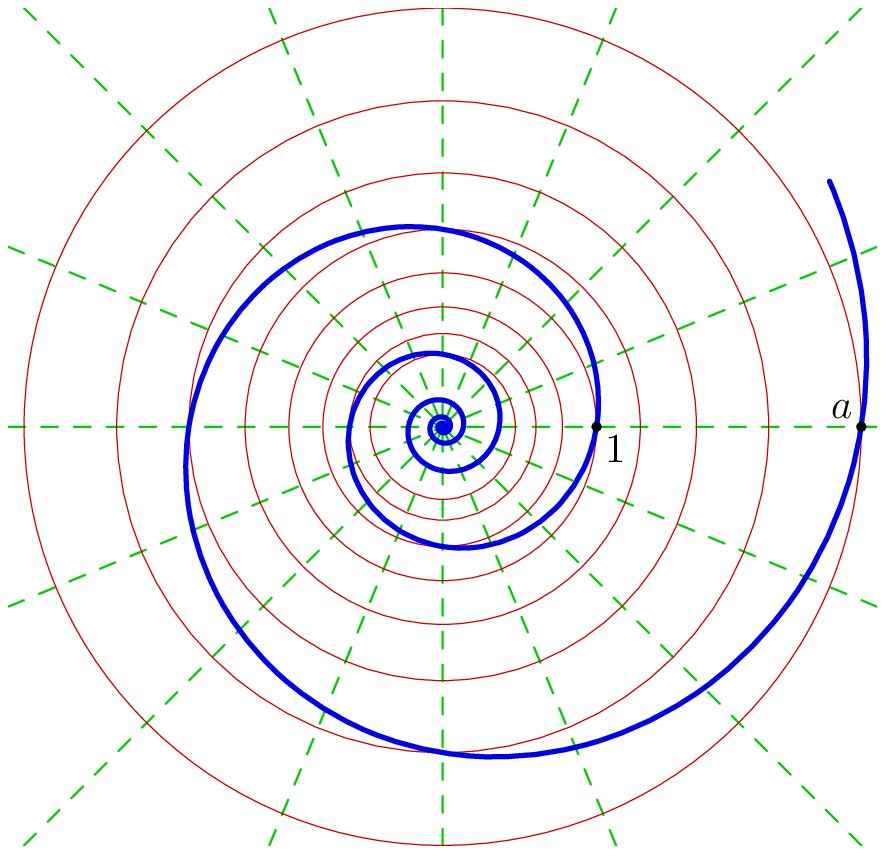}\hfill
   \includegraphics[scale=.7]{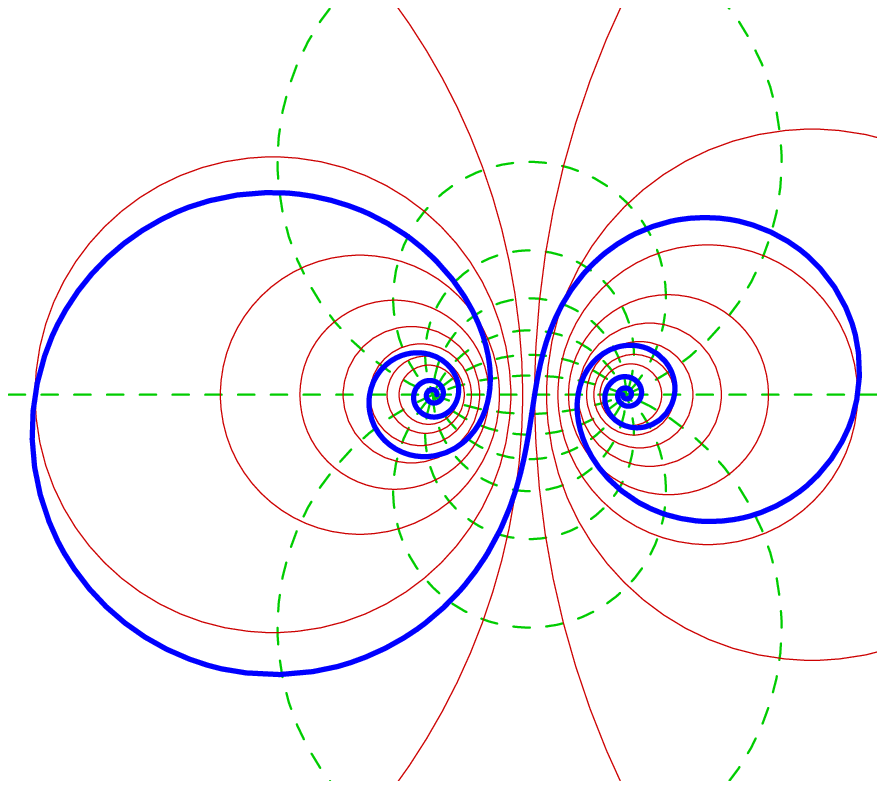}
  \caption[Loxodromes and pencils of cycles]{Logarithmic spirals
    (left) and loxodrome (right) with associated pencils of
    cycles. This is a combination of Figs.~\ref{fig:log-lox-init} and~\ref{fig:ellipt-hyperb-penc}.}
  \label{fig:log-spir-loxodr-pencils}
\end{figure}
In essence, the previous proposition says that a three-cycle and
\((\lambda,M)\) parametrisations are equivalent and delivers an
explicit procedure producing one from another. However, three-cycle
parametrisation is more geometric, since it links a loxodrome to a pair of
orthogonal pencils, see
Fig.~\ref{fig:log-spir-loxodr-pencils}. Furthermore, cycles \(C_1\),
\(C_2\), \(C_3\) (unlike parameters \(\lambda\) and \(M\)) can be
directly drawn on the plane to represent a loxodrome, which may be
even omitted.

\section{Applications of Three-Cycle Parametrisation}
\label{sec:applications}
Now we present some examples of the usage of three-cycle
parametrisation of loxodromes. First, we want to resolve
non-uniqueness in such parametrisations. Recall, that a triple
\(\{C_1,C_2,C_3\}\) is non-degenerate if \(C_2\neq C_3\) and \(C_3\)
is not zero-radius.
\begin{prop}
  Two non-degenerate triples \(\{C_1,C_2,C_3\}\) and
  \(\{C_1',C_2',C_3'\}\) parameterise the same loxodrome if and only
  if all the following conditions are satisfied:
  \begin{enumerate}
  \item \label{item:same-pencil}
    Pairs \(\{C_2,C_3\}\) and \(\{C_2',C_3'\}\)  span the same
    hyperbolic pencil. That is cycles \(C_2'\) and \(C_3'\) are linear
    combinations of \(C_2\) and \(C_3\) and vise versa.
  \item \label{item:same-lambda}
    Pairs \(\{C_2,C_3\}\) and \(\{C_2',C_3'\}\) define the same
    parameter \(\tilde{\lambda}\):
    \begin{equation}
      \label{eq:equal-lambdas}
      \nscalar {C_2}{C_3}=\nscalar {C_2'}{C_3'}.
    \end{equation}
  \item \label{item:ellipt-hyperb-ident}
    The elliptic-hyperbolic identity holds:
    \begin{equation}
      \label{eq:ellipt-hyperb-equat}
      \frac{\arccosh\nscalar {C_j}{C_j'}}{\arccosh\nscalar{C_2}{C_3}}
      \equiv
      \frac{1}{2\pi}\arccos\nscalar {C_1}{C_1'} \pmod{1}\,,
    \end{equation}
    where \(j\) is either \(2\) or \(3\).
  \end{enumerate}
\end{prop}
\begin{proof}
  Necessity of~\eqref{item:same-pencil} is obvious, since  hyperbolic
  pencils spanned by \(\{C_2,C_3\}\) and \(\{C_2',C_3'\}\) shall be
  both the image of concentric circles centred at origin under FLT
  \(M\) defining the loxodrome.  Necessity of~\eqref{item:same-lambda}
  is also obvious since \(\tilde{\lambda}\) is uniquely defined by the
  loxodrome. Necessity of~\eqref{item:ellipt-hyperb-ident} follows
  from the analysis of the following demonstration of sufficiency.

  For sufficiency, let \(M\) be FLT constructed through
  Procedure~\ref{proc:build-normalisation} from
  \(\{C_1,C_2,C_3\}\). Then~\eqref{item:same-pencil} implies that
  \(M(C_2')\) and \(M(C_3')\) are also circles centred at origin. Then
  Lem.~\ref{le:norm-prod-geometr} implies that the transformation
  \(D_{x+\rmi y}(1)\), where
  \(x=\arccosh\nscalar {C_2}{C_2'}\) and
  \(y=\arccos\nscalar{C_1}{C_1'}\) maps \(C_1'\) and \(C_2'\) to
  \(C_1\) and \(C_2\) respectively. Furthermore, from
  identity~\eqref{eq:equal-lambdas} follows that the same
  \(D_{x+\rmi y}(1)\) maps \(C_3'\) to \(C_3\). Finally,
  condition~\eqref{eq:ellipt-hyperb-equat} means that
  \(x+\rmi (y +2\pi n)= a(\tilde{\lambda} +2\pi \rmi)\) for 
  \(a=x/\tilde{\lambda}\) and some
  \(n\in\Space{Z}{}\). In other words
  \(D_{x+\rmi y}(1)=D_{\tilde{\lambda}}(a)\), thus \(D_{x+\rmi y}(1)\)
  maps SLS with the parameter \(\tilde{\lambda}\) to
  itself. Since \(\{M(C_1), M(C_2), M(C_3)\}\) and \(\{M(C_1'),
  M(C_2'), M(C_3')\}\)
  are two three-cycle parametrisations of the same SLS, \(\{C_1,C_2,C_3\}\) and
  \(\{C_1',C_2',C_3'\}\) are two three-cycle parametrisations of the
  same loxodrome.
\end{proof}
See Fig.~\ref{fig:equiv-param-loxodr} for an animated family of
equivalent three-cycle parametrisations of the same loxodrome (also
posted at~\cite{Kisil16a}).
Relation~\eqref{eq:ellipt-hyperb-equat}, which correlates elliptic and
hyperbolic rotations for loxodrome, regularly appears in this
context. The next topic provides another illustration of this.

\begin{proc}
  \label{proc:loxo-pass-point}
  To verify whether a loxodrome parametrised by three cycles
  \(\{C_1,C_2,C_3\}\) passes a point parametrised by a zero-radius
  cycle \(C_0\) we perform the following steps:
  \begin{enumerate}
  \item Define the cycle
    \begin{equation}
      \label{eq:Ch-cycle-for-point}
      C_h=tC_2+(1-t)C_3\,, \qquad
      \text{where } t=-\frac{\scalar{C_0}{C_3}}{\scalar{C_0}{C_2-C_3}}\,,
    \end{equation}
    which belongs to the hyperbolic
    pencil spanned by \(\{C_2,C_3\}\) and is orthogonal to
    \(C_0\), that is, passes the respective point.
  \item Find cycle \(C_e\) from the elliptic pencil orthogonal to
    \(\{C_2,C_3\}\) which passes \(C_0\). \(C_e\) is the solution of the
    system of three linear (with respect to parameters of \(C_e\))
    equations, cf. Ex~\ref{ex:cycle-product-geometric}:
    \begin{align*}
      \scalar{C_e}{C_0}&=0\,,\\
      \scalar{C_e}{C_2}&=0\,,\\
      \scalar{C_e}{C_3}&=0\,.
    \end{align*}
  \item  Verify the elliptic-hyperbolic relation:
    \begin{equation}
      \label{eq:ellipt-hyperb-equat-point}
      \frac{\arccosh\nscalar {C_h}{C_2}}
      {\arccosh\nscalar {C_2}{C_3}} \equiv
      \frac{1}{2\pi}\arccos\nscalar {C_e}{C_1} \pmod{1}\,.
    \end{equation}
  \end{enumerate}
\end{proc}
\begin{proof}
  Let \(M\) be the standard FLT associated  to   \(\{C_1,C_2,C_3\}\)
  from Procedure~\ref{proc:build-normalisation}.
  The point \(C_0\) belongs to the loxodrome if the transformation
  \(D_{\tilde{\lambda}}(t)\) for some \(t\) moves \(M(C_0)\) to the
  intersection \(M(C_1)\cap M(C_2)\). But \(D_{x+\rmi y}(1)\) with
  \(x= \arccosh\nscalar {C_h}{C_2}\) and
  \(y=\arccos\nscalar {C_e}{C_1}\) maps \(M(C_h)\rightarrow M(C_2)\) and
  \(M(C_e)\rightarrow M(C_1)\), thus it also maps \(M(C_0)\subset M(C_h)\cap M(C_e)\)
  to \(M(C_1)\cap M(C_2)\). Condition~\eqref{eq:ellipt-hyperb-equat-point}
  guaranties that
  \(D_{x+\rmi y}(1)=D_{\tilde{\lambda}}(x/\tilde{\lambda})\), as in
  the previous Prop.
\end{proof}
Our final example considers two loxodromes which may have completely
different associated pencils.
\begin{proc}
  \label{proc:intersection-angle}
  Let two loxodromes are parametrised by \(\{C_1,C_2,C_3\}\) and
  \(\{C_1',C_2',C_3'\}\). Assume they intersect at some point
  parametrised by a zero-radius cycle \(C_0\) (this can be checked by
  Procedure~\ref{proc:loxo-pass-point}, if needed). To find the angle of
  intersection we perform the following steps:
  \begin{enumerate}
  \item Use~\eqref{eq:Ch-cycle-for-point} to find cycles \(C_h\) and
    \(C_h'\) belonging to hyperbolic pencils, spanned by
    \(\{C_2,C_3\}\) and  \(\{C_2',C_3'\}\) respectively, and both passing
    \(C_0\).
  \item The intersection angle is
    \begin{equation}
      \label{eq:angle-between-loxodromes}
      \arccos\nscalar{C_h}{C_h'}-\arctan\left(\frac{\tilde{\lambda}}{2\pi}\right)
      +\arctan\left(\frac{\tilde{\lambda}'}{2\pi}\right)\,,
    \end{equation}
    where \(\tilde{\lambda}\) and  \(\tilde{\lambda}'\) are determined
    by~\eqref{eq:lambda-from-triple}. 
  \end{enumerate}
\end{proc}
\begin{proof}
  A loxodrome intersects any cycle from its hyperbolic pencil with the
  fixed angle \(\arctan(\tilde{\lambda}/(2\pi))\). This is used to
  amend the intersection angle \(\arccos\nscalar{C_h}{C_h'}\)
  of cycles from the respective hyperbolic pencils.
\end{proof}
\begin{cor}
  \label{cor:tangent-lox-circle}
  Let a loxodrome parametrised by \(\{C_1,C_2,C_3\}\) passes a point
  parametrised by a zero-radius cycle \(C_0\) as in
  Procedure~\ref{proc:loxo-pass-point}.  A non-zero radius cycle  \(C\) is
  tangent to the loxodrome at \(C_0\) if and only if two conditions
  holds:
  \begin{equation}
    \label{eq:tangent-lox-cycle}
    \begin{split}
      \scalar{C}{C_0}&=0\,,\\
      \arccos\nscalar{C}{C_h}&=\arctan\left(\frac{\tilde{\lambda}}{2\pi}\right)\,,
    \end{split}
  \end{equation}
  where \(C_h\) is given by~\eqref{eq:Ch-cycle-for-point} and is
  \(\tilde{\lambda}\) is determined by~\eqref{eq:lambda-from-triple}.
\end{cor}
\begin{proof}
  The first condition simply verifies that \(C\) passes \(C_0\),
  cf. Ex~\ref{ex:cycle-product-geometric}.  Cycle \(C\), as a
  degenerated loxodrome, is parametrised by \(\{C_e,C,C\}\), where
  \(C_e\) is any cycle orthogonal to \(C\) and \(C_e\) is not relevant
  in the following. The hyperbolic pencil spanned by two copies of
  \(C\) consists of \(C\) only. Thus we put \(C_h'=C\),
  \(\tilde{\lambda}'=0\) in~\eqref{eq:angle-between-loxodromes} and
  equate it to \(0\) to obtain the second identity
  in~\eqref{eq:tangent-lox-cycle}.
\end{proof}

\section{Discussion and Open Questions}
\label{sec:open-questions}

It was mentioned at the end of Section~\ref{sec:loxodr-encod-as} that
a three-cycle parametrisation of loxodromes is more geometrical than
their presentation by a pair \((\lambda,M)\). Furthermore, three-cycle
parametrisation reveals the natural analogy between elliptic and
hyperbolic features of loxodromes, see~\eqref{eq:ellipt-hyperb-equat}
as an illustration.  Examples in Section~\ref{sec:applications} show
that various geometrical questions are explicitly answered in term of
three-cycle parametrisation. Thus, our work extends the set of objects
in Lie sphere geometry---circle, lines and points---to the natural
maximal conformally-invariant family, which also includes
loxodromes. In practical terms, three-cycle parametrisation allows to
extend the library \texttt{figure} for M\"obuis invariant
geometry~\cite{Kisil14b} to operate with loxodromes as well.

It is even more important, that the presented technique is another
implementation of a general framework~\cites{Kisil15a,Kisil14b}, which
provides a significant advance in Lie sphere geometry. The Poincar\'e
extension of FLT from the real line to the upper half-plane was
performed by a pair of orthogonal cycles in~\cite{Kisil15a}. A similar
extension of FLT from the complex plane to the upper half-space can be
done by a \emph{triple of pairwise orthogonal cycles}. Thus, triples
satisfying FLT-invariant properties \eqref{it:orthogonal} and
\eqref{it:disjoint} of Prop.~\ref{pr:loxodrome-triple} present another
non-equivalent class of cycle ensembles in the sense
of~\cite{Kisil15a}. In general, Lie sphere geometry \emph{can be
  enriched by consideration of cycle ensembles} interrelated by a list
of FLT-invariant properties~\cites{Kisil15a}. Such ensembles become new
objects in the extended Lie spheres geometry and can be represented by
points in a \emph{cycle ensemble space}.

There are several natural directions to extend this work further, here are
just few of them:
\begin{enumerate}
\item Link our ``global'' parametrisation of loxodromes with
  differential geometry approach
  from~\cites{Bolt07a,Porter93a,Porter07a}. Our last
  Cor.~\ref{cor:tangent-lox-circle} can be a first step in this
  direction.
\item Consider all FLT-invariant non-equivalent classes of three-cycle
  ensembles on \(\Space{C}{}\): pairwise orthogonal triples
  (representing points in the upper half-space~\cite{Kisil15a}),
  triples satisfying properties \eqref{it:orthogonal} and
  \eqref{it:disjoint} of Prop.~\ref{pr:loxodrome-triple} (representing
  loxodromes), etc.
\item Extend this consideration for
  quaternions
  or Clifford algebras~\cites{GuerlebeckHabetaSproessig08,MoraisGeorgievSprossig14a}.
  The previous works~\cites{Porter98a,Porter98b} and availability of
  FSCc in this setup~\citelist{\cite{Cnops02a}*{Ch.~4}
    \cite{FillmoreSpringer90a}} make it rather promising. 
\item Consider M\"obius transformations in rings of dual and double
  numbers~\cites{Kisil12a,Kisil05a,Kisil15a,Kisil07a,Kisil09e,Kisil11a,%
    Mustafa17a,BarrettBolt10a}. There are enough indications that the
  story will not be quite the same as for complex numbers.
\item Explore further connections of loxodromes with
  \begin{itemize}
  \item   Carleson curves and
  microlocal properties of singular integral
  operators~\cites{BishopBoettcherKarlovichSpitkovsky99a,%
    BoettcherKarlovich01a,BarrettBolt08a}; or
  \item applications in operator
  theory~\cites{Simon11a,Vasilevski08a}.
  \end{itemize}
\end{enumerate}
Some combinations of those topics shall be fruitful as well.

\section*{Acknowledgments}
\label{sec:acknowledgments}
The second-named author was supported by the Summer UG Bursary Scheme
(2017) at the University of Leeds.  An earlier work~\cite{Hurst13a} on
this topic was supported by the same scheme in 2013.

Graphics for this paper was prepared using \texttt{Asymptote}
software~\cite{Asymptote}. We supply an \texttt{Asymptote} library
\texttt{cycle2D} for geometry of planar cycles under GPL
licence~\cite{GNUGPL} with the copy of this paper at
\href{http://arxiv.org/a/kisil_v_1}{arXiv.org}.

\label{se:appendix}

\begin{figure}[htbp]
  \centering
    \animategraphics[controls=true,width=.9\textwidth,poster=first]{50}{_loxodromes}{1}{200}
    \caption[Equivalent parametrisation of a loxodrome]{Animated
      graphics of equivalent three-cycle parametrisations of a
      loxodrome. The green cycle is \(C_1\), two red circles are
      \(C_2\) and \(C_3\).}
  \label{fig:equiv-param-loxodr}
\end{figure}

\small
\providecommand{\noopsort}[1]{} \providecommand{\printfirst}[2]{#1}
  \providecommand{\singleletter}[1]{#1} \providecommand{\switchargs}[2]{#2#1}
  \providecommand{\irm}{\textup{I}} \providecommand{\iirm}{\textup{II}}
  \providecommand{\vrm}{\textup{V}} \providecommand{\cprime}{'}
  \providecommand{\eprint}[2]{\texttt{#2}}
  \providecommand{\myeprint}[2]{\texttt{#2}}
  \providecommand{\arXiv}[1]{\myeprint{http://arXiv.org/abs/#1}{arXiv:#1}}
  \providecommand{\doi}[1]{\href{http://dx.doi.org/#1}{doi:
  #1}}\providecommand{\CPP}{\texttt{C++}}
  \providecommand{\NoWEB}{\texttt{noweb}}
  \providecommand{\MetaPost}{\texttt{Meta}\-\texttt{Post}}
  \providecommand{\GiNaC}{\textsf{GiNaC}}
  \providecommand{\pyGiNaC}{\textsf{pyGiNaC}}
  \providecommand{\Asymptote}{\texttt{Asymptote}}

\end{document}